\documentclass[12pt]{amsart}
\usepackage{amssymb}
\usepackage{graphicx,color}

\usepackage[all]{xy}
\usepackage{amscd}
\usepackage{hyperref}

\usepackage[active]{srcltx}

\nonstopmode

\def\cocoa{{\hbox{\rm C\kern-.13em o\kern-.07em C\kern-.13em o\kern-.15em A}}}

\DeclareMathOperator{\reg}{reg}
\DeclareMathOperator{\Hl}{H}
\DeclareMathOperator{\Ker}{Ker}
\DeclareMathOperator{\length}{length}

\newcommand {\PP}{\mathbb{P}}

\newcommand {\ZZ}{\mathbb{Z}}

\newtheorem{theorem}{Theorem}[section]
\newtheorem{lemma}[theorem]{Lemma}

\newtheorem{proposition}[theorem]{Proposition}
\newtheorem{corollary}[theorem]{Corollary} \newtheorem{definition}[theorem]{Definition} 
\newtheorem{remark}[theorem]{Remark}
\newtheorem{problem}[theorem]{Problem}
\newtheorem{question}[theorem]{Question}
\newtheorem{example}[theorem]{Example}

\newcommand{\red}[1]{{\color{red} \sf   [#1]}}

\numberwithin{equation}{section}
\title[On the intersection of ACM curves in $\PP^3$]{On the intersection of ACM curves in $\PP^3$}

\author[R.\  Hartshorne]{Robin Hartshorne}
\address{Department of Mathematics, University of California, Berkeley,  California, CA 94720-3840, USA}
\email{robin@math.berkeley.edu}

\author[R.\ M.\ Mir\'o-Roig]{Rosa M.\ Mir\'o-Roig${}^{*}$}
\address{Facultat de Matem\`atiques, Department d'\`Algebra i
 Geometria, Gran Via des les Corts Catalanes 585, 08007 Barcelona, Spain}
\email{miro@ub.edu}

\thanks{\noindent Printed \today \\
${}^{*}$  Part of the work for this paper was done while this
author was sponsored by MTM2010-15256.
}

\begin{document}

\subjclass[2010]{Primary 14C17, Secondary 14H45}
\keywords{Arithmetic genus, arithmetically Cohen-Macaulay, $h$-vector, biliaison type, numerical character}

\begin{abstract}    Bezout's theorem gives us the degree of intersection of two properly intersecting projective varieties.
As two  curves in $\PP^3$ never intersect properly, Bezout's theorem cannot be
directly used to bound the number of intersection points of such curves. In this work, we bound the maximum number of intersection points of two
integral ACM curves in $\PP^3$. The bound that we give is in many cases optimal as a function of only the degrees and the initial degrees of the
curves.
\end{abstract}

\maketitle

\tableofcontents

\section{Introduction}
In this paper we investigate the intersection of space curves. For varieties of complementary dimension in a projective space, their intersection is governed by Bezout's theorem: Thus two curves, of degrees $d$ and $e$, in the plane  intersect in $de$ points. Space curves do not ordinarily intersect. So we are led to pose the following question:

\begin{question} \rm Fixing some invariants of two (integral) curves  $C_1$ and $C_2$ in the projective 3-dimensional space $\PP^3$, what is the maximum number of intersection points of two such curves?
\end{question}

Since the genus of the union $C_1\cup C_2$ of two curves is determined by the genus of $C_1$ and $C_2$ individually and the number of their intersection points, our question is equivalent to

\begin{question} \rm Fixing some invariants of two  (integral) curves $C_1$ and $C_2$ in $\PP^3$, what is the maximum genus of the union of two such curves?
\end{question}
In this form our question is a generalization to reducible curves of the bounds known for irreducible curves by the work of many authors - the so-called Castelnuovo theory and the Halphen problem.

\vskip 2mm
In searching for answers to our questions, various other interesting questions arise. Is the maximum number of intersection points always attained when the two curves are in a common surface of the lowest degree that can contain both curves? If the maximum is attained, is the union of the two curves necessarily arithmetically Cohen-Macaulay? What can we say about the  set of points $T=C_1\cap C_2$ in  the case of  a maximum intersection?

\vskip 2mm
A complete answer to all these questions becomes quite complicated, depending on what is assumed about the initial curves $C_1$ and $C_2$. Therefore, we will pay special attention to situations in which restrictive hypotheses make possible a more concise answer. So for example if $C_1$ and $C_2$ are both complete intersection curves, a complete answer can be found by elementary means (see \S 2). If $C_1$ and $C_2$ are both
arithmetically Cohen-Macaulay (ACM for short) curves we can give good answers in many cases. The answers in general will fall into two parts: one is to establish an upper bound for the number of intersection points; the other is to ask whether this bound is actually attained for certain classes of curves.

\vskip 2mm
There seems to be scant attention to these questions  in the literature. If one of the curves is a line, we are asking for the maximum order of a multisecant line; this has been studied in various cases \cite{GP}, \cite{L}, \cite{N} and \cite{HS}. Giuffrida in \cite{G} and Diaz in \cite{D}
proved that the number  of intersection points of  two smooth non-planar irreducible
curves $C_1$ and $C_2$ in $\PP^3$ of degrees $d_1$ and $d_2$, respectively, is bounded by $(d_1-1)(d_2-1)+1$ and
the maximum is reached only if $C_1$ and $C_2$ are both on the same quadric surface.
 And a result of the second author
with Ranestad in \cite{MR-R} showed the existence of certain ACM curves with conjectured maximum order of intersection.

 \vskip 2mm
While many questions about space curves seem impossibly  complicated in general, there is the feeling that for ACM curves one should find reasonable answers. Thus the possible degrees, genus, postulation, and Hilbert schemes of ACM  curves are known, and depend only on  certain numerical invariants. For instance, the gonality of a general ACM  curve has been studied in \cite{HS}, the multisecant lines to ACM curves have been studied by Nollet in \cite{N} and Ellia has studied  the normal bundle to ACM  curves  in \cite{E}.

\vskip 2mm
Our motivation for this work was the hope that this  study of the  intersection of ACM  curves may help in finding the Gorenstein liaison class of finite sets of points in $\PP^3$ (cf. \cite{HSS}).

\vskip 2mm
Next we outline the structure of the paper. In section 2, we treat
the  case of complete intersection curves, where a complete answer can be obtained by elementary means. In section 3,
 we recall various numerical invariants associated to ACM curves, and we recall an important decomposition theorem (see Theorem \ref{decompThm}) for curves whose hyperplane section has a biliaison type with a gap. In section 4 we get bounds on the genus of the union of two ACM curves, which also give us bounds on their number of intersection points. For example we prove (see Theorem \ref{mainthm}) that if the biliaison character of the hyperplane section of $C_1\cup C_2$ has no gaps, then
 $$p_a(C_1\cup C_2)\le G_{CM}(d_1+d_2,max\{s_1,s_2\}).
 $$

 In section 5, we give some existence theorems for smooth curves and good surfaces that contain them. Then in section 6 we study linked curves, showing the existence of smooth linked curves with given $h$-vectors having the maximum number of intersection  points (see Theorem \ref{thm62}). This result enables us to prove  an old conjecture of the second author with Ranestad \cite[Conjecture 4.5 (a)]{MR-R}.

 In section 7 we consider "ordinary" ACM curves, those whose general hyperplane section consists of points in general position, and we compute the maximum number of intersection  points of two of them.

 We end with a short section of remaining open problems.

\vskip 2mm
 Throughout this paper we work  over an algebraically closed field of  arbitrary characteristic (except where otherwise noted). By the intersection of two curves $C_1$ and $C_2$ we mean the scheme-theoretic intersection $T=C_1\cap C_2$ and by the number of intersection points $\# (C_1\cap C_2)$ we mean the length of the zero-dimensional scheme $T$.


 \section{Complete intersection curves}

In this section we will consider the special case of complete intersection curves, where the results are elementary, to serve as an example and as a model for what we seek to achieve in more general cases.

    If $C$  is a complete intersection of two surfaces of degrees $s$ and $t$ in $\PP^3$, we will write $C=s \times t$ for short.

\begin{theorem} \label{ci}
Let $C_1$ and $C_2$ be distinct integral complete intersection curves $s_1\times t_1$ and $s_2\times t_2$. We assume $s_1\le t_1$, $s_2\le t_2$ and $s_1\le s_2$.
\begin{itemize}
\item[(a)] If $s_1=s_2=s$, then $\# (C_1\cap C_2)\le st_1t_2 $.
\item[(b)] If $s_1<s_2$ and $t_1\ge t_2$, then $\# (C_1\cap C_2)\le s_1s_2t_2 $.
\item[(c)] If $s_1<s_2$ and $s_2<t_1<t_2$, then $\# (C_1\cap C_2)\le s_1s_2t_1.$
\item[(d)] If  $s_1<s_2$ and $t_1\le s_2$,
then $\# (C_1\cap C_2)\le s_1t_1t_2$.
\end{itemize}
Furthermore, in each case the bounds are attained by smooth curves $C_1$, $C_2$ meeting transversally, and when they are, $C_1\cup C_2$ will be an ACM curve, and the intersection $T=C_1\cap C_2$ will be a complete intersection zero-dimensional scheme.
\end{theorem}
\begin{proof} (a) First we suppose that there is a common surface $S$ of degree $s$ containing both $C_1$ and $C_2$. Then $C_1 \sim t_1H$ and $C_2 \sim t_2H$ where $H$ is a hyperplane section of $S$. Thus the intersection number $C_1\cdot C_2$ is $t_1t_2 H^2=st_1t_2 $. This shows that the bound can be attained, and in this case $C_1\cup C_2=(t_1+t_2)H$
is a complete intersection $s\times (t_1+t_2)$, and the intersection $T=C_1\cap C_2$ is a complete intersection  $s\times t_1\times t_2$.
 Taking $S$ to be a smooth surface, and taking $C_1$ and $C_2$ general, we may assume that $C_1$ and $C_2$ are smooth, meeting transversally. Now suppose there is no such common surface $S$ of
degree $s$. Let $C_1\subset S$ and $C_2\nsubseteq S $. Then $C_1\cap C_2\subset S\cap C_2$ which has degree
$s(\deg(C_2))=s^2t_2$ by Bezout's theorem. Since $s\le t_1$, this is less than $st_1t_2$ and the first
case gives the maximum intersection.

\vskip 2mm
(b) Since $s_1<s_2$ and $C_2$ is irreducible, $C_2$ cannot be contained in a surface $S_1$ of degree $s_1$ containing $C_1$. Thus $C_1\cap C_2\subset S_1\cap C_2$ which has degree $s_1s_2t_2$. This proves the bound. If $t_2\le t_1$, this bound can be attained by choosing a surface $S$ of degree $t_1$ containing $C_2$. Then,
$C_1 \sim s_1H$ on $S$. So,  $C_1\cdot C_2=s_1(C_2 H)$ has degree $s_1(\deg(C_2))=s_1s_2t_2$. Since $C_2$ is a complete intersection $s_2\times t_2$, its ideal sheaf ${\mathcal I}_{C_2}$ is generated by global sections in degrees $\ge t_2$. Taking $C_2$ smooth, we can then find a smooth surface $S$ of degree $t_1$ containing $C_2$, and thus $C_1$ and $C_2$ smooth meeting transversally. In this case $C_1\cup C_2$ is obtained from $C_2$
 by a biliaison of height $s_1$ on $S$. Therefore it is ACM, but not necessarily a complete intersection. The intersection $T$ is however,
 a complete intersection $s_1\times s_2\times t_2$.

\vskip 2mm
(c) If $s_1<s_2<t_1<t_2$, then $C_1$ cannot be contained in a surface $S_2$ of degree $s_2$ containing $C_2$. So  $C_1\cap C_2\subset C_1\cap S_2$ which has degree $s_1s_2t_1$. This bound can be attained by taking a surface $S$ of degree $t_2$ containing both $C_1$ and $C_2$, in which case $C_2 \sim s_2H$. So,  $C_1\cdot C_2=s_2(\deg(C_1))=s_1s_2t_1$. In this case as in case (b), $C_1\cup C_2$ is ACM and $T$ is
 a complete intersection. In this case, as in (d) below, the existence of $C_{i}$ smooth is similar.

\vskip 2mm
(d) If there is a surface $S_2$ of degree $s_2$ containing $C_1$ and $C_2$, then $C_2 \sim t_2H$ on $S_2$. So,
$C_1\cdot C_2=t_2(C_1 H)$ has degree $t_2(\deg(C_1))=s_1t_1t_2$. If there is no such surface $S_2$, then $C_1\nsubseteq S_2$, so $C_1\cap C_2\subset C_1\cap S_2$ which has degree $s_1t_1s_2$ which is less than $s_1t_1t_2$. In the maximum case $C_1\cup C_2\sim C_1+t_2H$ is a biliaison of height $t_2$ from $C_1$ hence it is ACM.
\end{proof}

\begin{remark} \rm These results illustrate and suggest the following more general question: If $C_1$ and $C_2$ are ACM curves in $\PP^3$
with maximum number of intersection points, is the union $C_1\cup C_2$ necessarily an ACM curve? We will see that the answer is yes in many cases. On the other hand, it is rare that the intersection $T=C_1\cap C_2$ is a complete intersection scheme, but we can ask, what special properties does $T$ have? See discussion in section 8.
\end{remark}

 \begin{remark} \rm It is instructive to consider the case when $C_1$ is a line. In this case we are asking for the maximal order of a multisecant line to a complete intersection curve $C_2$. Theorem \ref{ci}(d) tells us that the maximum order of a multisecant line is $t_2$, which is consistent with Nollet's  determination of the maximum order of a multisecant line to any ACM curve (see \cite[Corollary 1.6]{N}). On the other hand, for a {\em general} complete intersection curve with  $s\ge 4$ (with few exceptions), the maximum order of a multisecant is 4 (see \cite[Theorem 1.4]{HS}). We can illustrate this result by an example. Take $C_1$ a line and $C_2$ a complete intersection $4\times 7$. Then $C_2$ is contained in a unique quartic surface $S$. If $C_2$ is general, the quartic surface must be also general. Since a general quartic  surface does not contain a line, the maximum number of intersection points is $L\cdot S=4$. On the other hand, if we take a special quartic surface containing  a line $L$, then $C_2$ will be $7H$ on $S$, and $L\cdot C_2=7=t_2$.
 So when we compute intersections of space curves in general, we should expect that the maximum intersection will be attained only by curves that are special in their Hilbert scheme.
\end{remark}



\section{Numerical invariants and the decomposition theorem}

In order to proceed, we need to make use of certain numerical invariants of ACM curves. In the literature there have been various different ways of  encoding this information: the numerical character of Gruson and Peskine \cite{GP}, the postulation character of \cite{MDP}, the $h$-vector \cite{MR} and the biliaison character $\lambda $ used in \cite{HS}. We will use the latter two in this paper, though all four systems can be easily translated from one to the other.

Given   a curve $C$ in $\PP^3$ with  homogeneous ideal $I_C$ and coordinate ring $R_C=k[x_0,\cdots ,x_3]/I_C$, we say that $C$ is arithmetically Cohen-Macaulay (ACM for short) if $R_C$ is a Cohen-Macaulay ring. We define the Hilbert function of an ACM curve $C$ in $\PP^3$ by $\Hl_C(\ell ) = \dim_k(R_C)_{\ell }$, and we define the $h$-vector of $C$ as
$h_C(\ell ) = \partial ^{2}\Hl_C(\ell )$, where $\partial $ is the difference function. If $Z$ is a 0-scheme in $\PP^2$, we define its $h$-vector analogously: $h_Z(\ell )=\partial \Hl_Z(\ell )$. It is clear that an ACM curve and its  general plane section  have the same $h$-vector.

\begin{definition}\rm \label{C2admissible}
A numerical function $h:\ZZ \longrightarrow \ZZ$  is  {\em C2-admissible} if it has the following properties for some integer $s\ge 1$:
$$\begin{cases}h(n)=0 \text{ for } n<0 \\
h(n)=n+1 \text{ for } 0\le n\le s-1, \\
h(n)\ge h(n+1) \text{ for } n\ge s-1, \\  h(n)=0 \text{ for } n\gg 0. \end{cases} $$

Furthermore,  $h$ is said to be of {\em decreasing type} if $h(a)>h(a+1)
$ for some $a$ implies  $h(n)>h(n+1)$ or $h(n)=0$ for all $n\ge a$.
\end{definition}

\begin{theorem}\label{h-vectorACM}  (a) If $C$ is an ACM curve in $\PP^3$, its $h$-vector is C2-admissible. Furthermore, every C2-admissible numerical function occurs as the $h$-vector of some ACM curve.

(b) If the ACM curve is integral, then its  $h$-vector is of decreasing type. Conversely,  if $h$ is a C2-admissible numerical function of decreasing type then there exists a smooth irreducible ACM curve $C\subset \PP^3$ with that $h$-vector.
\end{theorem}
\begin{proof} The results are well known and appear many times in the literature in different languages. See for example the book of Migliore \cite{Mi} for statements and further references.

As far as we can tell part (b) was first proved in \cite{GP} using the numerical character. In that language the condition that  an $h$-vector should be of decreasing type is equivalent to the condition that the numerical character should have no gaps.
\end{proof}

From now on, given any curve $C$ in $\PP^3$ we denote by $s(C)$ the least degree of a surface containing $C$, i.e. $$s(C)= inf \{\ell\in \ZZ \mid I_{C,\ell }\ne 0\}$$ and we call it the {\em initial degree} of $C$; we denote by $e(C)$ the {\em index of speciality} of $C$
$$e(C)= sup \{\ell\in \ZZ \mid \Hl^1(C,{\mathcal O}_C(\ell ))\ne 0\},$$ and we  denote by $t(C)$ the {\em second ideal degree} of $C$, namely
$$t(C)=sup\{ \ell\in \ZZ \mid  I_{C,\le \ell} \text{ is not principal} \}.$$
The fundamental numerical invariants of an ACM curve can be easily computed using the $h$-vector. In fact, we have (see, for instance, \cite{Mi})

\begin{proposition} \label{deg-genus} Let $C$ be an ACM curve in $\PP^3$ with $h$-vector $h(n)=c_n$, so we can write $h=\{c_0=1,c_1,c_2,\cdots , c_b \}$ where $b=sup\{ n\in \ZZ \mid h(n)>0 \}$. Then
 $$\begin{array}{rcl} \deg(C) & = & \sum _{i=0}^{b}c_i,\\
 p_a(C) & = & \sum _{i=2}^{b} (i-1)c_i, \\
 s(C) & = & inf \{n \in \ZZ \mid c_n<n+1 \}, \\
 t(C) & = & inf \{ n\ge s \mid c_n<s \} ,
 \\  e(C) & = & b-2 ,\\
 \reg(C) & = & b+1.\end{array}$$
\end{proposition}

We will also use the biliaison type $ \lambda$ of an ACM curve.

\begin{definition} \label{def-gap} \rm For any C2-admissible numerical function $h$ we define $$k_i=\# \{n \mid h_C(n)\ge s+1-i \} \text{ for } 1\le i \le s.$$
The sequence $\lambda =\{k_1,k_2,\cdots ,k_s\}$ is called the associated {\em biliaison type}.
\end{definition}

\begin{remark}\rm \label{d-g-lambda}
 The biliaison type gets its name from the property that an ACM curve $C$ in $\PP^3$ with biliaison type  $\lambda _C=\{k_1,k_2,\cdots ,k_s\}$ is obtained by a sequence of special biliaisons of height one from the empty curve, on surfaces of degrees $k_i$ \cite[Corollary 7.4]{HS}. In terms of the biliaison type $\lambda _C=\{k_1,k_2,\cdots ,k_s\}$ we have
$$
\begin{array}{rcl}
\deg(C) & = & \sum _{i=1}^sk_i, \\
p_a(C) & = & 1+\sum _{i=1}^s\frac{k_i(k_i-3)}{2}+\sum _{i=1}^s (s-i)k_i, \\
s(C) & = & s=\length (\lambda _C), \\
t(C) & = & s+k_1-1, \\
e(C) & = & k_s-3, \\
\reg(C) & = & k_s.
\end{array}
$$

In this language, Theorem \ref{h-vectorACM}  says that for any ACM curve $C$ in $\PP^3$, $\lambda _C=\{k_1,k_2,\cdots ,k_s\}$ is a strictly increasing set of positive  integers, and conversely, any such set of positive integers $k_1<k_2<\cdots <k_s$ occurs as  the biliaison type  $\lambda $ of some ACM curve $C$ in $\PP^3$.  We say that   $\lambda =\{k_1,k_2,\cdots ,k_s\}$  {\em has a gap}  if $k_{i+1}-k_i\ge 3$ for some $1\le i \le s-1$
The condition that an $h$-vector is of decreasing type is equivalent to saying that the biliaison type  $\lambda $ has no gaps.
\end{remark}

The study of ACM curves whose $h$-vector is not of decreasing type was started by Davis in \cite{Da}. In the language of the biliaison type  $\lambda $ his main result is

\begin{theorem}\label{davis} Let $C$ be an ACM curve in $\PP^3$ whose  biliaison type  $\lambda _C=\{k_1,k_2,\cdots ,k_s\}$ has a gap at $t$, so that $k_{t+1}\ge k_t+3$. Then $C$ has an ACM curve subcurve $D$ with $\lambda _D=\{k_{t+1},\cdots ,k_s\}$ and the residual curve $B$ is also ACM and has $\lambda _B=\{k_1,\cdots ,k_t\}$.

Furthermore, in this case $ \# (B\cap D)=\deg(B)·s(D)$, and also $\deg(B)·s(D)<\deg(D) ·s(B)$.
\end{theorem}
\begin{proof} For the existence of $B$ and $D$ see \cite{Da} or \cite[Proposition 7.18]{HS}  for an alternative proof.

For the second statement, using Remark \ref{d-g-lambda}, we write out the formulas for $p_a(B\cup D)$, $p_a(B)$, and $p_a(D)$ in terms of the $k_i$. A simple calculation shows that
$$ p_a(B\cup D)=p_a(B)+p_a(D)+(s-t)(\sum _{i=1}^tk_i)-1.
$$
Then by Lemma \ref{genus-union} below, $ \# (B\cap D)=(s-t)(\sum _{i=1}^tk_i)=\deg(B)·s(D)$.
Note also that
$$
(s-t)(\sum _{i=1}^tk_i)\le (s-t)tk_t<(s-t)tk_{t+1}\le t(\sum _{i=t+1}^sk_i),
$$
so $\deg(B)·s(D)<\deg(D)·s(B)$.
\end{proof}

 \begin{lemma} \label{genus-union} Let $C_1$ and $C_2$ be curves in $\PP^3$ with no common component. Then
 $$p_a(C_1\cup C_2)=p_a(C_1)+p_a(C_2)+\# (C_1\cap C_2)-1.$$
 \end{lemma}
 \begin{proof} See \cite[Proposition 4]{M1}.
 \end{proof}

Everything we have said so far has been for ACM curves in $\PP^3$, and the above results hold in arbitrary characteristic. When we consider curves in $\PP^3$ that are not necessarily ACM, the analogous results are more subtle, and their proofs often use a hypothesis of characteristic zero. For any curve $C$ in $\PP^3$, we consider a general hyperplane section $Z=C\cap H$. It is a zero-dimensional scheme in $\PP^2$, hence $Z$ is ACM and we can speak of the $h$-vector or the biliaison type of $Z$. A well-known result is

\begin{theorem}  If $C$ is an integral curve in $\PP^3$, then its general hyperplane section $Z$ has an $h$-vector of decreasing type.
\end{theorem}

\begin{proof} The result was proved by Gruson and Peskine in \cite{GP}. The result also follows (in characteristic zero) from the theorem of Harris that $Z$ has the Uniform position property \cite{Ha}.
\end{proof}

The result we will use most in the sequel is what happens in the case of curves whose general  hyperplane section has an $h$-vector not of decreasing type, that is to say, a biliaison type with a gap.

\begin{theorem} (Decomposition Theorem) (Char$(k)=0$). \label{decompThm} Let $C$ be a (locally Cohen-Macaulay) curve in $\PP^3$, and suppose that its general hyperplane section $Z$ has a  biliaison type $\lambda _Z=\{k_1,k_2,\cdots ,k_s\}$ with a gap at $t$, so that $k_{t+1}\ge k_t+3$. Then $C$ has a subcurve $D$ whose general hyperplane section $Z''$ has $\lambda _{Z''}=\{k_{t+1},\cdots ,k_s\}$. The residual curve $B$ of $D$ in $C$ then has general hyperplane section $Z'$ with $\lambda _{Z'}=\{k_{1},\cdots ,k_t\}$.
\end{theorem}
\begin{proof} This result is stated by Beorchia in \cite[Lemma 1.7]{B},  in the language of the numerical character. For the proof she refers to Strano \cite[Lemma 2]{S0}. A later paper of Strano \cite{S} states that his earlier proof of Lemma 2 was incorrect. He gives a  new proof using Davis's result (Theorem \ref{davis})  for the general hyperplane section $Z$ of $C$, then lifting the decomposition to $\PP^3$ using a result of Cook \cite[Proposition 10]{C}, whose proof is attributed to Green \cite{G} (see also \cite{C}).
\end{proof}

\begin{corollary} \label{lambdawithgap} (Char$(k)=0$) Let $C_1$ and $C_2$ be integral ACM curves in $\PP^3$, and let $C$ be the union $C_1\cup C_2$ (not necessarily ACM). Suppose that the biliaison type $\lambda _Z=\{k_1,k_2,\cdots ,k_s\}$ of the general hyperplane section $Z$ of $C$ has  a gap at $t$. Then (in one order or the other) $\lambda _{C_1}=\{k_{1},\cdots ,k_t\}$, $\lambda _{C_2}=\{k_{t+1},\cdots ,k_s\}$, and $s(C)=s(C_1)+s(C_2)$.
\end{corollary}

\begin{proof} According to the Decomposition Theorem, $C$ contains a subcurve $D$ and a residual curve $B$. Since $C$ is the union of two distinct irreducible ACM curves, we must have $C_1=B$, $C_2=D$ in one order or the other. Then $s(C_1)=t$, $s(C_2)=s-t$ and the initial degree $s(C)$ of $C$, which a priori maybe greater than $s(Z)=s$, is equal to $s$, because $C_1$ and $C_2$ are contained in surfaces of degrees $t$, $s-t$, respectively. Therefore, $C=C_1\cup C_2$ is contained in their union, a surface of degree $s$.
\end{proof}


\section{Bounds on the genus of reducible curves}

In this  section we derive some bounds  on the genus of space curves, generalizing the well known results  for integral curves. From these bounds we can then derive  bounds on the maximum number of intersection points of two ACM curves in $\PP^3$. Because of Lemma \ref{genus-union}, to bound the intersection number $\# (C_1\cap C_2)$ of two curves, it is equivalent to bound the genus of their union, $p_a(C_1\cup C_2)$. Therefore we will state results whichever way is most convenient.

\begin{lemma}\label{firstbound}  Let $C\subset \PP^3$ be a (locally CM) curve with Rao module $M=\oplus _{\ell }\Hl ^1(\PP^3,{\mathcal I}_C(\ell ))$ and  let $Z$ be its general plane section with $h$-vector $h_Z$. Let $C'\subset \PP^3$ be an ACM curve  with $h$-vector $h_{C'}=h_Z$ and let $N:=\Ker (M\stackrel{\times h}{\longrightarrow} M(1))$ where $h$ is a general linear form. Then $$ p_a(C) =p_a(C')-\lambda (N)$$ where $\lambda (N)$ is the length of $N$. In particular, $ p_a(C) \le p_a(C')$, with equality  if and only if $C$ is  ACM.
\end{lemma}
\begin{proof}We consider the exact sequence
$$0 \longrightarrow {\mathcal I}_C(n-1)\longrightarrow {\mathcal I}_C(n)\longrightarrow {\mathcal I}_{Z,\PP^2}(n)\longrightarrow 0.$$
Taking cohomology we get
$$0 \longrightarrow \Hl^0(\PP^3,{\mathcal I}_C(n-1))\longrightarrow \Hl^0(\PP^3,{\mathcal I}_C(n))\longrightarrow \Hl^0(\PP^2,{\mathcal I}_{Z,\PP^2}(n))\longrightarrow
 N_{n-1} \longrightarrow 0.$$
 Hence we have $$\partial h^0({\mathcal I}_C(n))=h^0({\mathcal I}_Z(n))-\lambda (N_{n-1}).$$
 Writing the postulation functions $$\begin{array}{rcl} \psi_C(n) & = & h^0({\mathcal O}_{\PP^3}(n))- h^0({\mathcal I}_C(n)) \\
 \psi_Z(n) & = & h^0({\mathcal O}_{\PP^2}(n))-h^0({\mathcal I}_{Z,\PP^2}(n)),
 \end{array}$$ we see
 $$ \partial  \psi_C(n)  = \psi_Z(n)+\lambda (N_{n-1}).$$
 But $C'\subset \PP^3$ is an ACM curve with $h$-vector $h_{C'}=h_Z$, so $\partial  \psi_{C'}(n)  = \psi_Z(n)$. Integrating for $n\gg 0$, we obtain
 $$   \psi_C(n)  = \psi_{C'}(n)+\lambda (N).$$
 Since for $n\gg 0$, we have  $\psi_C(n)=h^0({\mathcal O}_C(n))=\deg(C)n+1-p_a(C)$ and  $\psi_{C'}(n)=h^0({\mathcal O}_{C'}(n))=\deg(C')n+1-p_a(C')$, we get
 $$\deg(C)n+1-p_a(C)=\deg(C')n+1-p_a(C')+\lambda (N)$$ which together with the equality $\deg(C)=\deg(C')$ implies $$ p_a(C) =p_a(C')-\lambda (N).$$
 Finally, we observe that since $M$ is of finite length, we have $C$ is ACM if and only if $M=0$ if and only if $N=0$ if and only if $ p_a(C) =p_a(C')$.
 \end{proof}

 \begin{definition} \rm
Given integers $d$ and $s$, we define $G_{CM}(d, s)$ the maximum genus of an integral ACM
curve $C\subset \PP^3$ of degree $d$ not lying on a surface of degree $s-1$, if such curves exist, and 0 otherwise.
\end{definition}

\begin{remark} \label{remG_CM} \rm
Note by definition that for  $d$ fixed,  $G_{CM}(d,s)$ is a non-increasing function of $s$.
Given the formulas for $s$, $d$ and $g$ in terms of the $h$-vector (Proposition \ref{deg-genus}), it is a purely combinatorial task (valid in any characteristic) to compute the values of $G_{CM}(d,s)$ for all $d$, $s$. This has been done in \cite[Theorem 2.7]{GP}. There is one formula for the case $d>s(s-1)$ and another for the case $\frac{1}{2}s(s+1)\le d \le s(s-1)$. If $d<\frac{1}{2}s(s+1)$ there are no such ACM curves.
To find $G_{CM}(d,s)$ one must write an $h$-vector of decreasing type $$h: \ 1 \ 2 \ \cdots \ s \ a_1 \ a_2 \ \cdots \ a_r $$ of degree $d$. Since the  higher $a_j$ carry more height in the genus formula, one tries to make them as large as possible for higher $j$. Thus $ 1 \ 2 \ 3 \ 4 \ 3 $ and $1 \ 2 \ 3 \ 4 \ 2 \ 1$ both have $d=13$ and $s=4$ but the latter has he maximal genus.
\end{remark}

Now we can state our main theorem
\begin{theorem}\label{mainthm} (Char$(k)=0$) Let $C_1,C_2\subset \PP^3$ be  integral ACM curves of degrees $d_1$, $d_2$ and initial degrees $s_1$, $s_2$. Let $C=C_1\cup C_2$ and let $Z$ be a general hyperplane section of $C$.
\begin{itemize}
\item[(a)] If the biliaison type  $\lambda _Z$ of $C$ has no gaps, then
$$p_a(C)\le G_{CM}(d_1+d_2, max\{s_1,s_2\} ).$$
\item[(b)] If the biliaison type  $\lambda _Z$ of $C$ has  a gap, or if $s(C)=s_1+s_2$, then
$$\# (C_1\cap C_2)\le
min\{ d_1s_2,d_2s_1 \}.$$
\end{itemize}
Furthermore, in each case, if equality holds, then $C$ is ACM.
\end{theorem}

\begin{proof} (a) Let $C'$ be an ACM curve in $\PP^3$ with  biliaison type  $\lambda _{C'}=\lambda _Z$. Since $\lambda _Z$ has no gaps, we can take $C'$ to be an integral (even smooth) ACM curve in $\PP^3$ by Theorem \ref{h-vectorACM}. The initial degree $s'$ of $C'$ is the same as for $Z$, and $Z$ is the union of the hyperplane sections $Z_1$ and $Z_2$ of $C_1$ and $C_2$. Moreover, the initial degrees of $Z_1$ and $Z_2$ are $s_1$ and $s_2$ since $C_1$ and $C_2$ are ACM curves. Therefore, $s'\ge max\{s_1,s_2\}$. By definition, $p_a(C')\le G_{CM}(d_1+d_2,s')$ and since $G_{CM}(d,s)$ is a decreasing function of $s$ for $d$ fixed, we conclude  that  $p_a(C')\le G_{CM}(d_1+d_2,max\{s_1,s_2\} )$. Now, $p_a(C)\le p_a(C')$ by Lemma \ref{firstbound}, and this proves (a).

(b) If $\lambda _Z$ of $C$ has  a gap, we first apply Corollary \ref{lambdawithgap}  which tells us that (in one order or the other) $\lambda _{C_1}=\{k_{1},\cdots ,k_t\}$, $\lambda _{C_2}=\{k_{t+1},\cdots ,k_s\}$, $\lambda _{Z}=\{k_{1},\cdots ,k_s\}$ and there is a gap at $t$, namely $k_{t+1}\ge k_t+3$. Since $C_1$ and $C_2$ are  irreducible ACM curves, neither $\lambda _{C_1}$ nor $\lambda _{C_2}$ has a gap. In particular, $s(C)=s_1+s_2$.

Now assuming $s(C)=s_1+s_2$, let $S_1$ and $S_2$ be surfaces of degrees $s_1$ and $s_2$ containing $C_1$ and $C_2$ respectively. Then $C_1\cup C_2$ is contained in $S_1\cup S_2$, but $C_1$ is not contained in $S_2$, so $\#(C_1\cap C_2)\le \#(C_1\cap S_2)=d_1s_2$. Similarly $C_2$ is not contained in $S_1$, so $\#(C_1\cap C_2)\le \#(S_1\cap C_2)=d_2s_1$. Therefore, $\#(C_1\cap C_2)\le min(d_1s_2,d_2s_1)$.

To prove the last statement, if there is equality in (a), then $p_a(C)=p_a(C')$, and this implies that $C$ is ACM by Lemma \ref{firstbound}. If there is equality in (b), then  $C$ is ACM by the following Lemma \ref{ACM}.
\end{proof}

\begin{lemma} \label{ACM} Let $C_1,C_2\subset \PP^3$ be   ACM curves of degrees $d_1$, $d_2$ contained in surfaces $S_1$ and $S_2$ of degrees $s_1$ and $s_2$ such that $C_1\nsubseteq S_2$ and $C_2\nsubseteq S_1$. Assume $\#(C_1\cap C_2)= min(d_1s_2,d_2s_1)$. Then, $C_1\cup C_2$ is an ACM curve.
\end{lemma}

\begin{proof}  Interchanging indices if necessary, we may assume $\#(C_1\cap C_2)= d_1s_2$. Then clearly the intersection scheme $T=C_1\cap C_2$ is equal to $C_1\cap S_2$, so the ideal sheaf ${\mathcal I}_{T,C_1}\cong {\mathcal O}_{C_1}(-s_2)$. To show that $C_1\cup C_2$ is ACM, it will be sufficient to show that $\Hl^1(\PP^3, {\mathcal I}_{C_1\cup C_2}(m))=0$ for all $m\in \ZZ$. We consider the diagram of sheaves, for any $m\in \ZZ$,

$$
\begin{array}{cccccccccccccc}
 0 & \longrightarrow & {\mathcal I}_{C_1\cup C_2}(m)  & \stackrel{\alpha }{ \longrightarrow } &
 {\mathcal I}_{ C_2}(m)\stackrel{\beta }{ \longrightarrow } &
 {\mathcal I}_{C_2,C_1\cup C_2}(m) & \longrightarrow & 0 \\
 &&&& \uparrow &  \scriptstyle{\thickapprox }\uparrow \\
 &&&& {\mathcal O}_{\PP^3}(m-s_2)\stackrel{\gamma }{ \longrightarrow } & {\mathcal O}_{C_1}(m-s_2)  & \longrightarrow & 0
 \end{array}
$$
where the first vertical arrow comes from the inclusion of $ {\mathcal I}_{ S_2}$ in $ {\mathcal I}_{ C_2}$, and the second vertical arrow is the isomorphism ${\mathcal O}_{C_1}(m-s_2)\cong {\mathcal I}_{T,C_1}(m)\cong {\mathcal I}_{C_2,C_1\cup C_2}(m).$

Taking $\Hl^0$ of the terms in this sequence, $\Hl^0(\gamma )$ is surjective because $C_1$ is ACM. It follows that $\Hl^0(\beta )$ is surjective. On the other hand, $\Hl^1(\PP^3, {\mathcal I}_{ C_2}(m))=0$ because $C_2$ is ACM. Now it follows from the long exact cohomology sequence associated to the first row that $\Hl^1(\PP^3, {\mathcal I}_{C_1\cup C_2}(m))=0$ for all $m\in \ZZ$, so $C_1\cup C_2$ is ACM.
\end{proof}

\begin{remark} \rm It is worthwhile to point out that Theorem \ref{mainthm}(a) can be seen as a generalization of  Proposition 6.3 in \cite{HSS}.
\end{remark}

\begin{remark} \label{ConditionGap} \rm It is annoying to have a dichotomy between the biliaison type  $\lambda _Z$ having a gap or not in the statement, because we cannot tell what  $\lambda _Z$
is like a priori. However in applications we can often eliminate case (b) if $\lambda _{C_1}$ and $\lambda _{C_2}$ do not form subsets of a biliaison type with a gap.  In terms of the invariants $s_i$, $t_i$, $b_i$ of $C_1$ and $C_2$, using the formulas of Remark \ref{d-g-lambda}, the condition that $\lambda _{C_1}$ and $\lambda _{C_2}$ are the two subsets of a $\lambda $ with a gap is that $t_i-s_i\ge b_j+3$ for some choice of $i,j=1,2$.
\end{remark}

If we take into account the second ideal degrees $t_1=t(C_1)$ and $t_2=t(C_2)$, we can strengthen Theorem \ref{mainthm} in some cases.

\begin{theorem} \label{mainthm2} With the hypotheses  of Theorem \ref{mainthm}, assume furthermore that $s_1\ne s_2$ and $s_1,s_2<t_1,t_2$. If the biliaison type $\lambda _Z$ of $C$ has no gaps, then
$$p_a(C)\le G_{CM}(d_1+d_2,min(s_1+s_2,max(t_1,t_2))).$$
Furthermore,  if equality holds, then $C$ is ACM.
\end{theorem}
\begin{proof} As in the proof of Theorem \ref{mainthm}(a), it will be sufficient to show that $s'\ge min(s_1+s_2,max(t_1,t_2))$, where $s'$ is the initial degree of $Z=Z_1\cup Z_2$. Now $Z_{i}$ (for $i=1,2$) is a general hyperplane section of the integral ACM curve $C_{i}$ with initial degree $s_{i}<t_{i}$. Therefore, $C_{i}$ is contained in  a unique irreducible surface $S_{i}$ of degree $s_{i}$. Its hyperplane section will be an irreducible curve $D_{i}$ of degree $s_{i}$ containing $Z_{i}$. Since $C_{i}$ is an ACM curve, the $h$-vector of $Z_{i}$ is the same as $C_{i}$. It follows that any curve $E$ of degree less than $t_{i}$ containing $Z_{i}$ must contain $D_{i}$ as an irreducible component. Therefore, if $E$ is a curve of degree $s'$ containing $Z$, then either $s'\ge max(t_1,t_2)$ or $E$ must contain both $D_1$ and $D_2$ as irreducible components in which case $s'\ge s_1+s_2$. Now the result follows as in Theorem \ref{mainthm}(a).
\end{proof}

We give some examples to show that the bounds of Theorems \ref{mainthm} and \ref{mainthm2} are sometimes attained and sometimes not.

\begin{example} \rm
(a) Consider two complete intersection curves $C_1=s_1\times t_1$ and $C_2=s_2\times t_2$ with $s=s_1=s_2$. In this case, as we saw in Theorem \ref{ci} (a), the union $C_1\cup C_2$ is a complete intersection $s\times (t_1+t_2)$. This curve realizes the maximum genus $G_{CM}(s(t_1+t_2),s)$ and so gives equality in Theorem \ref{mainthm} (a).

\vskip 2mm
(b) Let $C_1$ and $C_2$ be complete intersections $2 \times 5$ and $3\times 3$. According to Theorem \ref{ci} (b), their union is obtained from $C_2$ by two biliaisons on $S_5$, so has $h$-vector: 1 \ 2 \ 3 \ 4 \ 5 \ 3 \ 1 . This realizes the maximum genus $G_{CM}(19,5)$. Since $5=s_1+s_2$, we have equality in Theorem \ref{mainthm2}. However this is much smaller than $G_{CM}(19,3)$ which is the bound  of Theorem \ref{mainthm} (a).

\vskip 2mm
(c) This time take complete intersections $2\times 6$ and $3\times 3$. Again, as in Theorem \ref{ci} (b), the union is obtained by biliaison from $3\times 3$ on $S_6$, and has $h$-vector: 1 \ 2 \ 3 \ 4 \ 5  \ 4 \ 2. This has genus strictly less than $G_{CM}(21,5)$, which is represented by 1 \ 2 \ 3 \ 4 \ 5 \ 3 \ 2 \ 1 , and much less than $G_{CM}(21,3)$, so neither Theorem \ref{mainthm} (a) nor Theorem \ref{mainthm2} is sharp.

\vskip 2mm
(d)  The bound of Theorem \ref{mainthm} (b)  can always be realized whenever $\lambda _Z$ has a gap. Just take an ACM curve with that $\lambda $-character. Then by Theorem \ref{davis} the intersection of the two subcurves attains the bound of Theorem \ref{mainthm} (b). If the two subcurves have $\lambda $-characters without gaps, then we can take the two curves $C_1$ and $C_2$ to be smooth \cite[Theorem 7.21]{HS}.

\vskip 2mm
(e) We give one more example to show that Theorem \ref{mainthm} (a) can actually fail, if $\lambda _Z$ has a gap. Let $C_1$ have $h$-vector 1 \ 2 and $C_2$  have $h$-vector 1 \ 1 \ 1 \ 1 \ 1 \ 1 \ 1 \ 1 . Then $C_1$ is a twisted cubic curve, with $d=3$, $p_a=0$, and $C_2$ is a plane octic curve with $d=8$, $p_a=21$. The twisted cubic can meet the plane of $C_2$ in at most 3 points; we can make $C_2$ pass through these three points, so the maximum possible intersection is $\#(C_1\cap C_2)=3$. Then $p_a(C_1\cup C_2)=23$, and $C_1\cup C_2$ is an ACM curve (Lemma \ref{ACM}) with $h$-vector 1 \ 2 \ 3 \ 1 \ 1 \ 1 \ 1
 \ 1 . On the other hand, Theorem \ref{mainthm} (a) would give the bound $G_{CM}(11,2)=20$, represented by the $h$-vector 1 \ 2 \ 2 \ 2 \ 2 \ 2 .
\end{example}

\section{Existence of good curves and surfaces}

Throughout this section, we assume $char(k)=0$ and we give some existence theorems for smooth curves  on integral surfaces that may have a finite number of singular points (we call then {\em surfaces with isolated singularities}).

\begin{proposition} \label{Bertini}  Let $C$ be a  curve in $\PP^3$ that has  embedding dimension $\le 2$ at almost all points. Let $m $ be an integer for which ${\mathcal I}_C(m )$ is generated by global sections. Then $C$ is contained in an integral surface of degree $m $ with isolated singularities.
\end{proposition}

\begin{proof} We consider the linear system $|{\mathcal I}_C(m)|$ of surfaces of degree $m$  containing $C$. Since ${\mathcal I}_C(m )$ is generated by global sections, the base locus of $|{\mathcal I}_C(m)|$ is just the curve $C$. Therefore, by  the characteristic zero Bertini theorem,  a general member of the linear system $|{\mathcal I}_C(m)|$   can have only singularities in $C$. Choose one point $P_i$ in each component of $C$, such that $C$ has embedding dimension $\le 2$ at $P_i$. Thus the general surface $X$ of the linear system $|{\mathcal I}_C(m)|$ can have only finitely many singular points. It follows that $X$ is integral and we are done.
\end{proof}

\begin{remark} \rm
If $C$ is a reduced  curve in $\PP^3$ and has  embedding dimension $\le 2$ at  all points, then we can take $X$ smooth. But in our applications, we cannot avoid considering curves that may have a non-reduced component, and in that case $X$ will necessarily have singularities for almost all degrees $m$.
\end{remark}

\begin{proposition}\label{omega-gbgs} Let  $C\subset \PP^3$ be a curve that has embedding dimension $\le 2$ at almost all points, and let $t=t(C)$. Then $\omega _C(3-t)$ is generated by global sections.
\end{proposition}

\begin{proof}  We proceed by induction on $s=s(C)$. If $s=1$, then $C$ is a plane curve of some degree $d$, and has $t=d$. In this case $\omega _C\cong {\mathcal O}_C(d-3)$, so $\omega _C(3-t)\cong {\mathcal O}_C$, which is generated by global sections.

Assume $s\ge 2$. Let $m=b(C)+1$, which is the index of regularity of $C$. Thus ${\mathcal I}_C(m)$ is generated by global sections. So, by Proposition \ref{Bertini}, $C$ is contained in a surface $X$ of degree $m$ with isolated singularities. On the surface $X$ we can write the exact sequence (see \cite[pg. 37]{GP} and \cite[Proposition 2.10]{GD})
$$ 0\longrightarrow {\mathcal O}_X\longrightarrow {\mathcal O}_X(C)\longrightarrow \omega_{C}(4-m)\longrightarrow 0.$$

Since $\omega _X\cong {\mathcal O}(m-4)$, twisting by ${\mathcal O}_X(-H)$, we obtain
$$ 0\longrightarrow {\mathcal O}_X(-H)\longrightarrow {\mathcal O}_X(C-H)\longrightarrow \omega_{C}(3-m)\longrightarrow 0.$$

Now $e=e(C)$ is equal to $b(C)-2$, which is $m-3$, so the term on the right is just $\omega _C(-e)$. By definition of $e$ and duality, $\Hl^0(\omega _C(-e))\ne 0$. Lifting a non-zero section of $\Hl^0(\omega _C(-e))$, we obtain a non-zero section of $\Hl^0({\mathcal O}_X(C-H))$. Since $X$ is integral, this gives us an effective generalized divisor $C'$ in the linear system $|C-H|$ on $X$. Then $C'$ is another ACM curve on $X$. By construction, it has $s'=s(C')<s$  and since it is contained in a surface with isolated singularities, it has embedding dimension $\le 2$ at almost all points. Therefore by the induction hypothesis  $\omega _{C'}(3-t')$  is generated by global sections, where $t'=t(C')$.

For $C'$ on $X$ we have the exact sequence
$$ 0\longrightarrow {\mathcal O}_X\longrightarrow {\mathcal O}_X(C')\longrightarrow \omega_{C'}(4-m)\longrightarrow 0.$$
Since $m=b(C)+1$, we have $t\le m$. On the other hand, looking at the $h$-vectors $h$ and $h'$ of $C$ and $C'$, we have
$$h'(n)=h(n)-1 \text{ for } 0\le n \le b,$$
so $t'=t-1$. Therefore, we have $3-t'\ge 4-m$ and twisting by some nonnegative integer $\delta \ge 0$ we find
$$ 0\longrightarrow {\mathcal O}_X(\delta H)\longrightarrow {\mathcal O}_X(C'+\delta H)\longrightarrow \omega_{C'}(3-t')\longrightarrow 0.$$
Now $\omega _{C'}(3-t')$ and ${\mathcal O}_X(\delta H)$ are both generated by global sections. So ${\mathcal O}_X(C'+\delta H)$ is also generated by global sections. But this sheaf is equal to ${\mathcal O}_X(C+(\delta -1)H )$, which maps surjectively to $\omega _C(3-t)$, and so this latter is also generated by global sections, as required.
\end{proof}

\begin{proposition}\label{C'_smooth}  Let  $C\subset \PP^3$ be an ACM curve contained in a surface $X$ of degree $m$ with isolated singularities. Assume that $C$ is smooth at any singular point of $X$ that it may contain, and assume that $m\le t(C)+1$. Then there is an irreducible smooth curve $C'$ in the linear system $|C|$ on $X$.
\end{proposition}

\begin{proof}  Since $C$ is contained in a surface $X$ with isolated singularities, it has embedding dimension $\le 2$ at almost all points. Therefore, by Proposition \ref{omega-gbgs}, $\omega _C(3-t)$ is generated by global sections. Since $m\le t+1$, it follows that $4-m\ge 3-t$, and so $\omega _C(4-m)$ is also generated by global sections. From the exact sequence
$$ 0\longrightarrow {\mathcal O}_X\longrightarrow {\mathcal O}_X(C)\longrightarrow \omega_{C}(4-m)\longrightarrow 0$$
we conclude that ${\mathcal O}_X(C)$ is also generated by global sections. Therefore the linear system $|{\mathcal O}_X(C)|$ on $X$ has no base points except possibly at the singular points of $X$ and at each of these $C$ is assumed to be smooth. Thus, by characteristic zero Bertini theorem, a general curve $C'\in |C|$ is smooth everywhere. Being an ACM curve it is also connected hence irreducible.
\end{proof}

\begin{remark} \rm Note the (possibly unexpected) consequence of the hypotheses of Proposition \ref{C'_smooth}, namely that the $h$-vector of $C$  must be of decreasing type since $C'$ is smooth and has the same  $h$-vector.  Note also that  since $C$ is contained in an irreducible surface of degree $m$, we must have $m=s$ or $m\ge t$. Hence there are really only three possibilities for $m$, namely, $m=s$, $m=t$, or $m=t+1$.
\end{remark}

\begin{proposition} \label{technicalresult} Let $C\subset \PP^3$ be a curve, not necessarily ACM, contained in a surface $X$ of degree $s=s(C)$ with isolated singularities. Then for any $m\ge t(C)$, the curve $C$ is also contained in a surface $X'$ of degree $m$ with isolated singularities.
\end{proposition}

\begin{proof} By definition of $s$ and $t$, the curve $C$ must be contained in a complete intersection $C'$ of type $s\times t$, for some surface of degree $t$. Since $C'$ is contained in $X$ with isolated singularities, it follows that $C'$ has embedding dimension $\le 2$ at almost all points. Also clearly, ${\mathcal I}_{C'}(m)$ is generated by global sections for any $m\ge t$. Then by Proposition \ref{Bertini}, $C'$ is contained in a surface $X'$ of degree $m$ with isolated singularities. Now $C$ is contained in $C'$, so $C$ is also contained in $X'$.
\end{proof}

\begin{remark} \rm Even if $C$ is nonsingular, we cannot be sure that $X'$ can be taken to be nonsingular. The trouble is that $C'$ contains another piece, the curve linked to $C$ by $C'$, over which we have no control. If it has a non-reduced component, then $X'$ can be singular. This  is the reason why we have allowed isolated singularities throughout this section.
\end{remark}

\begin{proposition}\label{S-smooth} If $h$ is an $h$-vector of decreasing type, then there exists a smooth ACM curve $C$ with $h$-vector $h$ lying on a surface $X$ of degree $s=s(h)$ having isolated singularities.
\end{proposition}

\begin{proof} We proceed by induction on the degree. Let $s=s(h)$. Let $h'=h-H_s$.

If $t>s$, then $s(h')=s$ and by induction there is a smooth ACM curve $C'$ with $h$-vector $h'$ on a surface $X_s$ of degree $s$ with isolated singularities. Consider $C=C'+H$ on $X_s$. By  Proposition \ref{C'_smooth} since $m=s<t$, there is a smooth curve $C\sim C'+H$ on $X_s$.

If $t=s$, then $s(h')=s-1$ and $t(h')=s-1$ or $s$, since $h$ had decreasing type. Now by induction we find a smooth ACM curve $C'$  on a surface $X_{s-1}$ of degree $s-1$ with isolated singularities. By  Proposition \ref{technicalresult}, since $t(h')=s-1$ or $s$, $C'$ is also contained in a surface $X_s$ of degree $s$ with isolated singularities. Then as before, there is a smooth curve $C\sim C'+H$ on $X_s$.
\end{proof}

\begin{remark} \rm
 In fact one can require the surface $X$ in Proposition \ref{S-smooth} to be smooth. See for example \cite{N}, Proposition 2.6 and the proof of Theorem 3.3.
\end{remark}

\section{Linked curves}

 If $C_1$ and $C_2$ are two ACM curves in $\PP^3$ linked by a complete intersection $m\times n$ ($m\le n$),  then their $h$-vectors satisfy the well-known relationship
$$h_2(\ell )=h_{m, n}(\ell )-h_1(m+n-2-\ell)$$
for each $\ell \in \ZZ$, where $h_{m, n}$ is the $h$-vector of the complete intersection, which is
$$h_{m, n}(\ell )=\begin{cases} \ell + 1 & \text{ for }  0\le \ell \le m-1, \\
 m & \text{ for }  m-1\le \ell \le n-1, \\
 m+n-1- \ell & \text{ for }  n-1\le \ell \le m+n-1, \\
 0 & \text{ otherwise. }
\end{cases}$$

\begin{lemma} \label{tec111} Let the $h$-vectors $h_1$, $h_2$ be linked by $h_{m,n}$ with $m\le n$, and assume that $m=s_1=s(h_1)$. Then
\begin{itemize}
\item[(a)] $b_2\le n-2$,
\item[(b)] If furthermore $s_2<s_1$, then $n\le b_1+1$, so $b_2<b_1$, and also $t_2\le m$.
\end{itemize}
\end{lemma}

\begin{proof}
(a) Since $m=s_1$, clearly $b_2\le b(h_{m,n})-m=(m+n-2)-m=n-2$.

(b) Since $s_2<s_1=m$, clearly $h_2(\ell )<m$ for all $\ell $. Taking $\ell =m-1$, we find $h_2(m-1)=h_{m,n}(m-1)-h_1(n-1)$. Since $h_{m,n}(m-1)=m$, we find $h_1(m-1)>0$, so $n-1\le b_1$. In particular, using part (a), we get $b_2<b_1$.

To find $t_2$, if  $m=n$, then $b_2\le m-2$, so $t_2\le m-1   $. If $n>m$, we compute $h_1(n-2)\ge b_1-n+3$ since $h_1$ has decreasing type. Then $h_2(m)\le m+n-_1-3<s_2$. Therefore, $t_2\le m$.
\end{proof}

\begin{theorem} \label{thm62} (a) If $C_1$ and $C_2$ are integral ACM curves that are linked by a complete intersection  $m\times n$ with $m=s_1\le n$, then the union $C_1\cup C_2$ gives the maximum possible intersection of integral ACM curves with $h$-vectors equal to those of $C_1$ and $C_2$.

(b) Given $h$-vectors of decreasing type $h_1$, $h_2$ linked by a complete intersection $m\times n$, with $m=s_1\le n$, there exist smooth ACM curves $C_1$, $C_2$ lying on  a surface $X$ of degree $m$ with isolated singularities and such that $C_2\sim nH-C_1$ on $X$, so that $C_1$ and $C_2$ are linked by a complete intersection $m\times n$.
\end{theorem}

\begin{proof} (a) First note that since $C_1\cup C_2$ is a complete intersection, its genus is equal to the bound $G_{CM}(mn,m)$. Now according to Theorem \ref{mainthm}, if $C$ is any union of $C_1$ and $C_2$, then $p_a(C)\le G_{CM}(mn,m)$ unless possibly the lambda character $\lambda _Z$ has a gap. The condition for having a gap (see Remark \ref{ConditionGap}) is that for some choice of $i,j=1,2$, $i\ne j$, we have $t_i-s_i\ge b_j+3$.

If $s_2<s_1$, then by Lemma \ref{tec111}, $b_2<b_1$, so there is only the possibility $t_1-s_1\ge b_2+3$. If $s_1=s_2$, then by interchanging $C_1$ and $C_2$ if necessary we may still assume that $t_1-s_1\ge b_2+3$.

Now compare the $h$-vector corresponding to the union $C=C_1\cup C_2$ where $\lambda _Z$ has a gap to the $h$-vector of the complete intersection $m\times n$. Both can be regarded as being $h_1$ increased by a total of $\deg(C_2)$ in various degrees. In case $\lambda _Z$ with a gap, the terms of $h_2$ are added to those of $h_1$ in degrees $<t_1$. In the case of the complete intersection, they are all added in degrees $\ge t_1$. So clearly the genus of the latter is greater. Hence the complete intersections realize the maximum intersection of $C_1$ and $C_2$.

(b) Given $h_1$ and $h_2$ linked by a complete intersection $m\times n$ with $m=s_1\le n$, we use Proposition \ref{S-smooth}
to show the existence of a smooth ACM curve $C_2$ lying on a surface $X_{s_2}$ with isolated singularities.

If $s_2<s_1$, then $t_2\le m$ by Lemma \ref{tec111}. So by Proposition \ref{technicalresult}, $C_2$ also lies on a surface $X_m$ of degree $m$ with isolated singularities. If $s_2=s_1$, then $C_2$ already lies on such a surface.

Now by Lemma \ref{tec111}, $b_2\le n-2$. But $b_2+1$ is the regularity, so ${\mathcal I}_{C_2}(m)$ is generated by global sections in $\PP^3$. Since ${\mathcal I}_{C_2}(m)$ maps surjectively to ${\mathcal I}_{C_2,X}(m)$, this later also is globally generated by global sections. But this is isomorphic to ${\mathcal O}_X(nH-C_2)$ on $X$. By the characteristic 0 Bertini theorem, there is a smooth curve $C_1$ in  the linear system $|nH-C_2|$. Then $C_1$ and $C_2$ are the required curves.
\end{proof}

\begin{remark} \rm  Since we are looking for curves with maximum intersection, we restrict to the case  $m=s(C_1)$. Curves linked on surfaces of degrees $m>s_1,s_2$ may have maximum intersection on that surfaces, but not maximum intersection for those $h$-vectors. For example, $h_1:1 \ 2 \ 1$ and $h_2: 1 \ 2 \ 2$ are linked by $h_{3,3}: 1 \ 2 \ 3 \ 2 \ 1 $, but their maximum intersection is achieved on a quadric surface, with union $h: 1 \ 2 \ 2\ 2 \ 2$.
\end{remark}

We can now use Theorem \ref{thm62} to prove a conjecture of the second author with Ranestad \cite[Conjecture 4.5 (a)]{MR-R}.

\begin{theorem} \label{conjectureMR-R}
For any positive integers $s\le t$ there exist smooth ACM curves $C_s$ and $C_t$ with $h$-vectors $1 \ 2 \ \cdots s$ and $1 \ 2 \ \cdots \ t$, lying on a surface  $X_t$ of degree $t$ with isolated singularities, and meeting transversally in
$$\#(C_s\cap C_t)= {s+1\choose 2}t-{s+1\choose 3}$$
points. This is the maximum possible intersection for ACM  curves with these $h$-vectors.
\end{theorem}

\begin{proof}
Note: the existence of smooth ACM curves $C_s$, $C_t$ and the computation of the intersection number was done by another method in \cite{MR-R}. The proof of the maximality is new here.

We proceed by induction  on $t-s\ge 0$. If $s=t$, then $C_s$ and $C_t$ are linked by $s\times (s+1)$, so the result is a consequence of Theorem \ref{thm62}.
The computation of the intersection number, for any $s$, $t$, is straightforward using Proposition \ref{deg-genus} and Lemma \ref{genus-union}.

Now suppose $t>s$. By induction there exist smooth ACM curves $C_s$ and $C_{t-1}$ lying on a surface $X_{t-1}$ as above, with $h$-vectors
$ 1 \ 2 \ \cdots s$ and $1 \ 2 \ \cdots \ t-1 $, and such that $C_s\cup C_{t-1}$ has $h$ vector $1 \ 2 \ \cdots \ t-1 \ s \ \cdots \ 2 \ 1 $. Then by Proposition
 \ref{technicalresult}, the union $C_s\cup C_{t-1}$ also lies on a surface $X_t$ of degree $t$ with isolated singularities. By Proposition \ref{C'_smooth} there is a smooth ACM curve $C_t\thicksim C_{t-1}+H$ on $X_t$. Then $C_s$ and $C_t$ have the required properties.

 Since the $h$-vector  $1 \ 2 \ \cdots \ t \ s \ \cdots \ 2 \ 1$ realizes the maximum genus $G_{CM}(d_s+d_t,t)$, the intersection is maximum, by Theorem \ref{mainthm}.
 \end{proof}

\begin{remark} \rm By an analogous  method we can find the maximum number of intersection points of  two smooth ACM curves $C_t^d$ and $C_s^d$ in $\PP^3$ defined by matrices with entries forms of degree $d$ proving Conjecture 4.5(b) of \cite{MR-R}.
\end{remark}

\begin{proposition} \label{Tlinked}
If $C_1$ and $C_2$ are ACM curves linked as in Theorem \ref{thm62}, or if they are $C_s$ and $C_t$  as in Theorem \ref{conjectureMR-R}, then the intersection $T=C_1\cap C_2$ is of the form $rH-K$ on $C_1$ or $C_2$, and hence is an arithmetically Gorenstein zero-scheme in $\PP^3$.
\end{proposition}

\begin{proof}
It is well known that the intersection of linked ACM curves has the form $rH-K$ and is arithmetically Gorenstein. (See for example \cite[Proposition 4.2.1]{Mi} or \cite[Proposition 1.3.7]{M}). In the second case, we have only to observe that in the induction step, $C_s\cap C_t$ is obtained by one ascending biliaison on $C_s$ from $C_s\cap C_{t-1}$, hence is again of the form $rH-K$ with $r$ replaced by $r+1$.
\end{proof}


\section{Ordinary ACM curves}

In this section, we will determine the maximum number of intersection points of ACM curves with certain classes of $h$-vectors (the so-called {\em ordinary} $h$-{\em vectors}), and we will prove the existence of irreducible nonsingular {\em ordinary} ACM curves realizing these intersections. To use the results of section 5 we need to consider curves on integral surfaces that may have a finite number of singular points. We assume $char(k)=0$ throughout this section.

\begin{definition}\label{def_ordinary} \rm
 Following the terminology introduced by Gruson et al. in \cite{GHL}, we will say that an ACM curve $C\subset \PP^3$
 (resp. its $h$-vector)  is {\em ordinary} if its $h$-vector is the $h$-vector of a set of general points in $\PP^2$, which means that its $h$-vector must be of  the  form $1,2,\cdots , s-1, s, a $ for $s\ge 1$ and $0\le a\le s$.
\end{definition}

\begin{theorem}\label{mainthmordinarycurves} Suppose given two ordinary $h$-vectors $h_1$ and $h_2$ with the same initial degree $s$, say $h_1=1,2, \cdots,s,a$ and $h_2=1,2,\cdots,s,b$. Then there exist smooth ACM curves $C_1$ and $C_2$ in $\PP^3$, meeting transversally, and having the maximum possible number of intersection points for ACM curves with the given $h$-vectors $h_1$ and $h_2$. Furthermore, the union $C_1\cup C_2$ is ACM and is contained in a surface $X$ of degree $s$ with isolated singularities.

If we  restrict the problem by requiring that the union of the two curves be contained in an irreducible surface of degree $s+1$, then there are other smooth ACM curves $C_1'$ and $C_2'$ with $h$-vectors $h_1$ and $h_2$, and having the maximum possible number of intersection points for curves with these $h$-vectors and subject to the above restriction. Again the union $C_1'\cup C_2'$ is ACM and is contained in a surface $X'$ of degree $s+1$  with isolated singularities.

Furthermore, in both cases the 0-dimensional schemes $T=C_1\cap C_2$ and $T'=C_1'\cap C_2'$ are strongly glicci, namely, they can be obtained by ascending Gorenstein biliaisons from complete intersections.

The $h$-vectors $h_3$ of $C_1\cup C_2$ and $h_3'$ of $C_1'\cup C_2'$, from which one can compute the actual number of intersection points, are given as follows (let $c=a+b$):

\begin{itemize}
\item[(i)] If $c=0$, then $h_3=h_3'=1 \ 2 \ \cdots \ s-1 \ s \ s \ s-1 \ \cdots \ 2 \ 1 $
\item[(ii)] If $0<c\le s$, then $h_3= 1 \ 2 \ \cdots \ s-1 \ s \ s \ s\ s-1 \cdots \ \widehat{s-c} \ \cdots \ 2 \ 1$, and

    \hskip 3cm $h_3'= 1 \ 2 \ \cdots \ s-1 \ s \ s+1 \ s\ s-1 \cdots \ \widehat{s+1-c} \ \cdots \ 2 \ 1$.
\item[(iii)] If $s<c\le 2s$, then $h_3= 1 \ 2 \ \cdots \ s-1 \ s \ s \ s \ s\ s-1 \cdots \ \widehat{2s-c} \ \cdots \ 2 \ 1$, and

    \hskip 3cm $h_3'= 1 \ 2 \ \cdots \ s-1 \ s \ \ s+1 \ s+1 \ s\ s-1 \cdots \ \widehat{2s+2-c} \ \cdots \ 2 \ 1$
\end{itemize}
where $\widehat{x}$ means omit the number $x$.
\end{theorem}

\begin{proof} We observe at the outset that each of the $h$-vectors $h_3$ and $h_3'$ in the list represents the maximal genus $G_{CM}(d,s)  $ or $G_{CM}(d,s+1)$  for their degree (cf. Remark \ref{remG_CM}), so by Lemma \ref{genus-union} and Theorem \ref{mainthm}, these curves do give the maximum number of intersection points possible, subject to the restriction of being contained in an irreducible surface of degree $s+1$ for $h_3'$.

To show the existence of smooth curves  whose union has these $h$-vectors, we proceed by induction on $s$. If $s=1$, then each of the $h$-vectors $h_1$, $h_2$ is either $1$ or $1 \ 1$. In this case the results are obvious: from $1$ and $1$ we obtain $1 \ 1$, from $1$ and $1 \ 1$ we obtain  $1 \ 1 \ 1$ or $1 \ 2$; from $1 \ 1$ and $1 \ 1$ we obtain $1 \ 1 \ 1 \ 1$ or $1 \ 2 \ 1$. The intersections $T=C_1\cap C_2$  in each case are complete intersections, consisting of $1$, $2$, or $4$ points.

So now let us consider $s\ge 2$, assuming that (i), (ii) and (iii) have been established for $s-1$.

\vskip 2mm
\underline{Case (i):} $c=0$.  In this case $a=b=0$, so the result follows from Theorem \ref{conjectureMR-R}. Note that the intersection $T$ is arithmetically Gorenstein, which in codimension  three implies strongly glicci \cite{MN}.

\vskip 2mm
\underline{Case (ii):} $0<c\le s$. We apply the induction hypothesis to $1 \ 2 \ \cdots \ s-1 \ a$ and $1 \ 2 \ \cdots \ s-1 \ b$. Using the second case $h_3'$ of (ii) (or of (iii) when $a+b=s$) we find smooth ACM curves $D_1$ and $D_2$ with these $h$-vectors,  whose union $D_1\cup D_2$ is ACM, contained in a surface $X$ of degree $s$
with isolated singularities, and with $h$-vector $1 \ 2 \  \cdots \ s-1 \ s \ s-1 \ \cdots \ \widehat{s-c} \ \cdots \ 1$. Let $C_1$ and $C_2$ be general elements of the linear systems $|D_1+H|$ and $|D_2+H|$ on $X$. By Proposition \ref{C'_smooth}, we can take both $C_1$ and $C_2$ smooth, meeting transversally. The $h$-vector of $C_1\cup C_2$ is obtained from that of $D_1\cup D_2$ by adding two hyperplanes $H_X$, giving $h_3= 1 \ 2 \  \cdots \ s \ s \ s \ \cdots \ \widehat{s-c} \ \cdots \ 1$. The intersection $T=C_1\cap C_2$ is obtained by two ascending Gorenstein biliaison from $D_1\cap D_2$, hence is strongly glicci.

For the restricted case, to find $h_3'$, we instead apply  induction hypothesis to  $1 \ 2 \ \cdots \ s-1 \ a $
and  $1 \ 2 \ \cdots \ s-1 \ b-1$ (assuming wlog $a\le b$, hence $b>0$). We find $D_1$ and $D_2$ smooth ACM curves whose union $D_1\cup D_2$ has $h$-vector
 $1 \ 2 \ \cdots \ s-1 \ s  \ s-1 \ \cdots \ \widehat{s+1-c} \ \cdots \ 1$ and is contained in a surface $X$ of degree $s$ with isolated singularities. Now we take $C_1\in |D_1+H_X|$ on $X$ smooth, meeting $D_2$ transversally, and with $C_1\cup D_2$ having $h$-vector $ 1 \ 2 \ \cdots \ s-1 \ s \ s \ s-1 \ \cdots \widehat{s+1-c} \ \cdots \ 1$. We now apply Proposition \ref{technicalresult} to $C_1\cup D_2$ to show that it is contained in a surface $X'$ of degree $s+1$ with isolated singularities. Then take $C_2\in |D_2+H|$ on $X'$. As before $C_1$ and $C_2$ will meet transversally, and their union will be ACM with $h$-vector $h_3'=1 \ 2 \ \cdots \ s \ s+1 \ s  \ \cdots \ \widehat{s+1-c} \ \cdots \ 1$, as required. Again $T=C_1\cap C_2$  is obtained from $D_1\cap D_2$ by two ascending biliaisons.

\vskip 2mm
\underline{Case (iii):} $s<c\le 2s$ is similar. If $a$, $b$ are both $\le s-1$, we apply the induction hypothesis to $1 \ 2\ \cdots \ s-1 \ a$ and $1\ 2 \ \cdots \ s-1 \ b$, obtaining  $D_1$, $D_2$ smooth ACM curves with $D_1\cup D_2$ having $h$-vector $1\ 2 \ \cdots \ s  \ s\ \cdots \widehat{2s-c} \ \cdots \ 1$ in a surface $X$ of degree $s$ with isolated singularities. Adding back $H_X$ twice we obtain $1\ 2 \cdots \ s\ s\ s\ s\cdots \ \widehat{2s-c} \ \cdots \ 1$.

If however $a=s$ and $b\le s-1$ (or vice versa),  we apply the induction hypothesis to $1 \ 2\ \cdots \ s-1 $ and $1\ 2 \ \cdots \ s-1 \ b$.  This put us back in Case (ii), where we find $1 \ 2 \ \cdots \ s \ \cdots \widehat{s-b} \ \cdots \ 1$ in $X$ of degree $s$. Adding back $H_X$ three times, we obtain $1\ 2 \ \cdots \ s\ s\ s\ s \cdots \ \widehat{s-b} \cdots \ 1$. Since $s-b=2s-c$ we are done.

If $a=b=s$, we apply Case (i) already proved above to $1\ 2 \ \cdots \ s$ and $1 \ 2 \ \cdots \ s$ obtaining $1\ 2 \ \cdots \ s\ s \cdots \ 2\ 1$. Then adding back two copies of the hyperplane section $H_X$ we get $1\ 2\ \cdots \ s\ s\ s\ s\ \cdots \ 2\ 1$.

To obtain $h_3'$, we apply induction to  $1\ 2 \ \cdots \ s-1\ a-1$ and $1 \ 2 \ \cdots \ s-1 \ b-1$. This gives $1\ 2 \ \cdots \ s\ s \cdots \ \widehat{2s+2-c} \ \cdots \ 2\ 1$ on $X$ of degree $s$. By Proposition \ref{technicalresult} this is also contained in an $X'$ of degree $s+1$. Then we add back two hyperplane sections in $X'$ to obtain $1\ 2\ \cdots \ s\ s+1\ s+1\ s\ \cdots \widehat{2s+2-c} \ \cdots  \ 2\ 1$.
\end{proof}

\begin{remark} \rm
If we consider ordinary $h$-vectors $h_1$ and $h_2$ with ·$s_1<s_2$, we can always reduce to the case $s_1=s_2$ by successively subtracting hyperplanes from $h_2$. Thus the same methods  as in the Theorem \ref{mainthmordinarycurves}  will produce the maximum  number of intersection points and the $h$-vector $h_3$ of the union. However, the conclusions are more complicated, so we do not state them explicitly here. One difference is that the $h$-vector $h_3$ of the union may no longer represent $G_{CM}(d_1+d_2,s_2)$, so that we need a different argument to prove maximality. Another difference is that the result may depend on $a$ and $b$ individually, and not just their sum $c$. We illustrate these points with some examples.
\end{remark}

\begin{example} \rm
(a) Let $h_1: \ 1$  and $h_2: \ 1 \ 2 \ 3 \ 4 \ 4$. We subtract two copies of $H_4$, and one each of $H_3$ and $H_2$ from $h_2$ to arrive at $h_2': \ 1$. The union of $h_1$ and $h_2'$ is $1 º 1 $. Adding back $H_2$, $H_3$ and two copies of $H_4$ we find for the union $h_3: \ 1 \ 2 \ 3 \ 4 \ 4 \ 1$. This does not represent $G_{CM}(15, 4)$, but it is the maximum number of intersection points, since it comes from one biliaison from $1 \ 2 \ 3 \ 4 \ 1 $ which represents $G_{CM}(11,4)$. This confirms Nollet's theorem (\cite[Corollary 1.6]{N}) that the maximal multisecant to the curve $C_2$ is of order 5.

\vskip 2mm
(b) Let $h_1: \ 1 \ 2 $ and $h_2: \ 1 \ 2 \ 3 \ 4 \ 5 \ 5 $. We reduce to $h_2': \  1  \  2 $. Then the union $h_1$ and $h_2'$ is $1 \ 2 \ 2 \ 1$. Adding back $H_3$, $H_4$ and two copies of $H_5$ we get $h_3: \ 1 \ 2 \  3 \ 4 \ 5 \ 5 \ 2 \ 1 $ for the union. This does not represent $G_{CM}(23,5)$ but it is maximal for the same reason as in (a) above: at the  previous step we have $1 \  2 \ 3 \ 4 \ 5 \ 2 \ 1$ which does represent $G_{CM}(18,5)$.

\vskip 2mm
(c) Let $h_1: \ 1 \ 2 \ 1$ and $h_2;  1 \ 2 \ 3 \ 4 \ 5 \ 4$. We reduce to $h_2': \ 1 \ 2 $. The union of  $h_1$ and $h_2'$ is $1 \ 2 \ 2 \ 2$ according to Theorem \ref{mainthmordinarycurves}. Now we need to add back $H_3$, but this union is not contained in an irreducible cubic surface. Therefore we must use the case $h_3'$ of Theorem \ref{mainthmordinarycurves} (ii) giving the union $1 \ 2 \ 3 \ 1$. Then we can add back $H_3$, two copies of $H_4$ and $H_5$ to get $h_3: \ 1 \ 2 \  3 \ 4 \ 5 \ 4 \ 3 \ 1 $ for the union which  does  represent $G_{CM}(23,5)$ and so  is maximum.

Note that examples (b) and (c) both have $s_1=2$, $s_2=5$ and $a+b=c=5$, however, the answers $h_3$ are different, hence dependent also on $a$ and $b$.

\end{example}


\section{Open problems}

While we have given bounds on the number of intersection points of two ACM curves in $\PP^3$, we have not found the exact maximum in all cases. So there remains

\begin{problem} \rm \begin{itemize}
\item[(a)] Given two $h$-vectors $h_1$ and $h_2$ of decreasing type, find the maximum number of intersection points of two integral ACM curves $C_1$ and $C_2$ with these $h$-vectors.
\item[(b)] When the maximum is realized,
\begin{itemize}
\item[(b1)] is $C_1\cup C_2$ an ACM curve?
\item[(b2)] is $C_1\cup C_2$ contained in an integral surface of least possible degree that could contain $C_1$ and $C_2$?
\item[(b3)] is the intersection $C_1\cap C_2$ strongly glicci?
\end{itemize}
 \end{itemize}
  \end{problem}

\begin{remark} \rm
This problem could be solved by an algorithm similar to the inductive method used in section 7.  For the induction we will need to solve also the following restricted problems. Given $h$-vectors $h_1$ and $h_2$ of decreasing type, and given an integer $m$ such that for each $i=1,2$, either $m=s_i$
or $m\ge t_i$, solve the restricted problem of finding the maximum number of intersection points of integral ACM curves $C_1$ and $C_2$ with $h$-vectors $h_1$ and $h_2$, subject to the condition that $C_1\cup C_2$ be contained in an integral surface $X$ of degree $m$. To do this, make an induction: whenever there is a curve $C_2'\thicksim C_2-H$ on $X$, use by induction the solution of the problem $h_1$, $h_2'$, $m$. Note that $C_1\cdot C_2$ is maximum on $X$ if and only if $C_1\cdot C_2'$ is maximum on $X$. The difficulty is to show that an inductive step is always possible with $C_2'$ integral and that it will lead to a new smooth ACM curve $C_2$. We give an example to illustrate the process.
\end{remark}

\begin{example} \rm
Let  $h_1: \ 1 \ 2 \ 3 \ 4$ and $h_2: \ 1 \ 2 \ 3 \ 4 \ 3 \ 2 \ 1$, and $m\ge 4$.

\vskip 2mm
(a) If we take $m=4$, the solution is simple. $C_1$ lies on a smooth surface $X$ of degree 4. On that surface we can realize $C_2$ as $4H$. So the intersection number is $4\deg(C_1)=40$ and is realized by $h_3:  \ 1 \ 2 \ 3 \ 4 \ 4 \ 4 \ 4 \ 4$ which does not represent $G_{CM}(26,4)$ but gives the
the maximum intersection of $C_1$ and $C_2$ also without restriction. Since $C_1\cup C_2$ is also contained in an integral surface of degree 8, the answer will serve for any $m\ge 8$.

\vskip 2mm
(b) Suppose we require $m=5$. We can place both curves on a surface $X_5$ of degree 5. Then $C_2\thicksim L+3H$, where $L$ is a line in $X_5$, with $h_2':  \ 1 $, The union of $h_1$ and $h_2'$ will be $1 \ 2 \ 3 \ 4 \ 1$. Adding back $3H$ on $X_5$ gives $h_3: \ 1 \ 2 \ 3 \ 4  \ 5 \ 5 \  5 \ 1$.  One must verify that the linear system $|L+3H|$ on $X_5$ contains a smooth curve $C_2$. This answer serves for $m=5$ and $m\ge 7$.

\vskip 2mm
(c) If we require $m=6$, put $C_2$ in a surface $X_6$ of degree 6. Subtract two hyperplanes to get $h_2': \ 1 \ 2 \ 1 $. This reduces to the new problem $h_1: \ 1 \ 2 \ 3 \ 4 $ and $h_2': º 1 \ 2 \ 1 $ with $m=6$. These are ordinary ACM curves, whose maximum intersection is ACM with $h_3': \ 1 \ 2 \ 3 \ 4 \ 3 \ 1$.
Adding back $2H$ on $X_6$ gives $h_3: \ 1 \ 2 \ 3 \ 4 \ 5 \ 6 \ 4 \ 1 $. Again we must verify that one can obtain a smooth curve $C_2$ in this way, and that the intersection is maximum (which it is since the intersection of $C_1$ and $C_2'$ is).
\end{example}

\begin{remark} \rm
Concerning Problem 8.1 (b3), it is worthwhile maintaining the distinction between glicci and strongly glicci, because there are examples of zero-schemes in $\PP^3$ that are glicci but not strongly glicci \cite{EHS}, and it is still unknown whether every zero-scheme in $\PP^3$ is glicci. We have seen that for complete intersections $C_1$, $C_2$ the intersection $T$ is also a complete intersection (Theorem \ref{ci}). In the case of linked curves, $T$ is arithmetically Gorenstein (Proposition  \ref{Tlinked}), and for ordinary ACM curves, the intersection $T$ is at least strongly glicci (Theorem \ref{mainthmordinarycurves}). We should perhaps add that in this last case, $T$ need not be arithmetically Gorenstein. Indeed, the maximum intersection of curves with $h$-vectors $h_1: \ 1 \ 2 \ 1$ and $h_2: \ 1 \ 2 \ 2$  on a quadric surface is a set of ten points, and there is no non planar arithmetically Gorenstein set of 10 points in $\PP^3$
\end{remark}
\vskip 4mm
\noindent  {\bf Acknowledgements:}   This material is based upon work supported by the National Science Foundation under Grant No. 0932078 000, while the second author was in residence at the Mathematical Science Research Institute (MSRI) in Berkeley, California, during the Commutative Algebra Program, 2012-13. We would like to thank the MSRI
  for kind
 hospitality.

%
%

\end{document}